\documentclass[12pt,a4paper]{amsart}

\usepackage[left=25truemm,right=25truemm,top=30truemm,bottom=30truemm]{geometry}
\usepackage{mathtools}
\usepackage{amssymb}
\usepackage{amscd}
\usepackage{thmtools}
\usepackage{mathrsfs}
\usepackage{graphicx}
\usepackage{subcaption}
\usepackage{float}
\usepackage[english]{babel}
\usepackage[dvipsnames]{xcolor}
\usepackage{framed}
\usepackage{setspace}
\usepackage{enumitem}
\usepackage[all]{xy}
\usepackage{tikz}
\usetikzlibrary{trees}
\usepackage{listings}
\usepackage{hyperref}
\usepackage{cleveref}

\theoremstyle{plain}
\newtheorem{thm}{\textrm{Theorem}}[section]
\newtheorem*{thm*}{\textrm{Theorem}}
\newtheorem{lem}[thm]{\textrm{Lemma}}
\newtheorem{cor}[thm]{\textrm{Corollary}}
\newtheorem{prop}[thm]{\textrm{Proposition}}
\newtheorem*{prop*}{\textrm{Proposition}}
\newtheorem*{cor*}{\textrm{Corollary}}

\theoremstyle{definition}
\newtheorem{example}[thm]{\textrm{Example}}
\newtheorem{defi}[thm]{\textrm{Definition}}
\newtheorem{rem}[thm]{\textrm{Remark}}

\newcommand{\I}{\newblock{\rm{I}}}
\newcommand{\II}{\newblock{\rm{II}}}

\newcommand{\Int}{\newblock{\rm{Int}}}

\title{\small Non-singular extensions of horizontal stable fold maps from surfaces to the plane}

\author{Koki Iwakura}

\date{}

\begin{document}

\maketitle
\vspace{-5pt}

\begin{abstract}
In this paper, we study the non-singular extension problem of horizontal stable fold maps. 
This problem asks what conditions ensure the existence of a submersion whose restriction to the boundary coincides with a given map, called a non-singular extension. 
By defining a combinatorial object called a pairing map, we prove that the existence of a non-singular extension is equivalent to the existence of a pairing map. 
Furthermore, to facilitate the application of the main theorem, we compute the Euler characteristics and the fundamental groups of compact $3$-dimensional manifolds that serve as the source manifolds of non-singular extensions. 
\end{abstract}
\vspace{-5pt}

\section{Introduction}
In the field of \emph{singularity theory} of differentiable maps, we investigate the properties of singular points of smooth maps between smooth manifolds.  
This field traces its origins to the pioneering works of Morse in the 1940s and Whitney and Thom in the 1950s.
\vspace{1pt}

The theory has expanded and revealed properties of singular points of maps. 
In particular, the study of how singular points influence the global structure of manifolds is known as \emph{global singularity theory}.
This area has been significantly developed by researchers such as Saeki, Sakuma, Sz\H{u}cs, and others, yielding notable results and insights.
\vspace{1pt}

In this paper, we focus on a specific global property of maps, called ``non-singular extendability."
This is a property of a map that can be extended to a submersion, which is the simplest object in the sense of a singularity, on a manifold of one higher dimension. 
The problem of finding necessary and sufficient conditions for maps is called the \emph{non-singular extension problem}. 
For Morse functions, this problem has been considered in several papers up to the present day. 
However, the non-singular extension problem for stable maps, which are generalizations of Morse functions, has not been considered until now. 
In this paper, we propose a method to determine the non-singular extendability of horizontal stable fold maps (\textrm{Definition\;2.2}) from a closed orientable surface to $\mathbb{R}^2$ in \textrm{Theorem\;4.1}. 
Before providing the contents of this paper, we introduce some previous research on the non-singular extension problem for Morse functions. 
\vspace{1pt}

The non-singular extension problem for real-valued Morse functions of $1$-dimensional manifolds was first studied by Blank--Laudenbach\;\cite{BL}. 
Their paper initiated further research on the generalization of source manifolds as follows.
The problem for Morse functions of closed orientable surfaces was studied by Curley\;\cite{C}.
Curley used a combinatorial object called an allowable collapse to state the existence of non-singular extensions. 
Curley proved that the existence of an allowable collapse is equivalent to the existence of a non-singular extension. 
Furthermore, Curley also investigated detailed conditions of an allowable collapse under which a Morse function of $S^{2}$ can be extended to a submersion of $D^{3}$. 
Laroche\;\cite{L} studied the problem for Morse functions of closed non-orientable surfaces and derived a necessary and sufficient condition by extending Curley's definition of allowable collapses to ensure compatibility with Laroche's condition.
Laroche also studied the detailed conditions of an allowable collapse under which a Morse function of a Klein bottle can be extended to a submersion of a solid Klein bottle. 
Furthermore, Iwakura\;\cite{Iwakura1} studied the non-singular extension problem for circle-valued Morse functions of closed orientable surfaces. 
\vspace{1pt}

Blank, Laudenbach, Curley, Laroche, and the author studied the problem of Morse functions of low dimensional manifolds. 
The non-singular extension problem for Morse functions of higher dimensional source manifolds was also investigated. 
Barannikov\;\cite{B} examined conditions under which Morse functions of spheres with arbitrary dimension can be extended to a submersion of disks and obtained necessary conditions regarding combinatorial properties of Morse complexes.
Furthermore, Seigneur\;\cite{S} studied the same problem and obtained necessary conditions in terms of algebraic properties of Morse complexes.
\vspace{1pt}

We briefly introduce the structure of this paper. 
In Section\;2, we provide several preliminary notions, including the definition of a combinatorial object called a pairing map (\textrm{Definition\;2.13}), which is essential for stating our main result, \textrm{Theorem\;4.1}. 
In Section\;3, we provide examples of pairing maps. 
In Section\;4, we state and prove \textrm{Theorem\;4.1}, which establishes that a non-singular extension of a given map exists if and only if a pairing map exists. 
In Section\;5, we compute the Euler characteristics and the fundamental groups of compact orientable $3$-dimensional manifolds with non-empty boundary that serve as the source manifolds of non-singular extensions.
\vspace{1pt}

Throughout this paper, manifolds and maps between them are assumed to be of class $C^{\infty}$ unless otherwise specified.

\section{Preliminaries}
In this section, we will prepare some necessary notations to state \textrm{Theorem\;4.1}, which is the main theorem of this paper.  
We assume that $\mathbb{R}^{2}$, as the target manifold, is given a right-handed (or the standard) orientation and that oriented manifolds with boundary induce the orientation on the boundary by the outward first convention.  
\vspace{1pt}

To introduce a horizontal stable fold map, we first introduce several preliminary notions. 
Let $M$ be a closed surface. 
We denote by $C^{\infty}(M,\mathbb{R}^{2})$ the set of smooth maps of class $C^{\infty}$ from $M$ to $\mathbb{R}^{2}$, equipped with the Whitney $C^{\infty}$ topology (see \cite{GG} for the definition).
For $f\in C^{\infty}(M,\mathbb{R}^{2})$, we define the subset $S(f)$ of $M$ by
\begin{equation*}
S(f)=\{p\in M~|~\textrm{rank}\;df_{p}<2\},
\end{equation*}
which is called the \emph{singular point set} of $f$, and a point of $S(f)$ is called a \emph{singular point} of $f$.
A point $p\in S(f)$ is called a \emph{fold point} of $f$ if there exist local coordinates $(x_{1},x_{2})$ of $M$ around $p$, and local coordinates $(y_{1},y_{2})$ of $\mathbb{R}^{2}$ around $f(p)$ such that $f$ is written in these charts as 
\begin{equation*}
(x_{1},x_{2})\mapsto (y_{1},y_{2})=(x_{1}^{2},x_{2}).
\end{equation*}
A map whose singular point set consists only of fold points is called a \emph{fold map}.
\vspace{1pt}

For maps $f, f'\in C^{\infty}(M,\mathbb{R}^{2})$, we say that $f$ is \emph{right-left equivalent} to $f'$ if there exist diffeomorphisms $\Phi\colon M\to M$ and $\Psi\colon\mathbb{R}^{2}\to\mathbb{R}^{2}$ such that $f'=\Psi\circ f\circ\Phi^{-1}$. 
A map $f$ is called a \emph{stable map} if there exists an open neighborhood $U_{f}$ of $f$ in $C^{\infty}(M,\mathbb{R}^{2})$ such that $f$ is right-left equivalent to each element in $U_{f}$ (see \cite{GG} for more details). 
In particular, stable maps having only fold points as the singular points are called \emph{stable fold maps}.

\begin{rem}
For a stable fold map $f$, the restriction $f|_{S(f)}$ to the singular point set $S(f)$ is an immersion with normal crossings. 
For more details, see \cite{GG}. 
\end{rem}

\begin{defi}[\textbf{Horizontal}]
Let $M$ be a closed surface and $f\colon M\to \mathbb{R}^{2}$ be a map. 
Then, $f$ is said to be \emph{horizontal} if $f$ is injective on each component of $M\setminus S(f)$.
In particular, if $f$ is a stable fold map, then we call $f$ a \emph{horizontal stable fold map}.
\end{defi}

We next introduce the precise definition of non-singular extensions. 

\begin{defi}[\textbf{Non-singular extension}]
Let $M$ be a closed orientable surface and $g\colon M\times[0,1)\to\mathbb{R}^{2}$ be a submersion such that $g|_{M\times\{0\}}$ is a stable map.
We assume that there exist a compact orientable $3$-dimensional manifold $N$ without closed components and a submersion $F\colon N\to\mathbb{R}^{2}$ that makes the following diagram commutative:
\[
  \xymatrix{
    M\times[0,1) \ar[r]^{~~~~g} \ar[d]_{i\;} &\mathbb{R}^{2} \\
    N, \ar[ru]_{F}
  }
\]
where $i$ is a collar of $\partial N$ identifying $M\times\{0\}$ with $\partial N$.
Then, $F$ is called a \emph{non-singular extension} of $g$. 
\end{defi}

\begin{rem}
In the above definition, we assume that the map $g$ from $M\times[0,1)$ to $\mathbb{R}^2$ is given. 
This is because, when considering extensions of a given map from closed surfaces to the plane, it is first necessary to extend to the collar neighborhood of the boundary. 
Since maps of closed surfaces into the plane can often be extended to a submersion on the product of surfaces and $[0,1)$, we also assume that $g$ has no singular points. 
For example, a fold map from a closed surface to $\mathbb{R}^2$ has an immersion lift by Haefliger's lifting theorem\;\cite{Hae}; thus, such an extension is possible. 
\end{rem}

\begin{figure}[t]
\centering
\includegraphics[width=130mm]{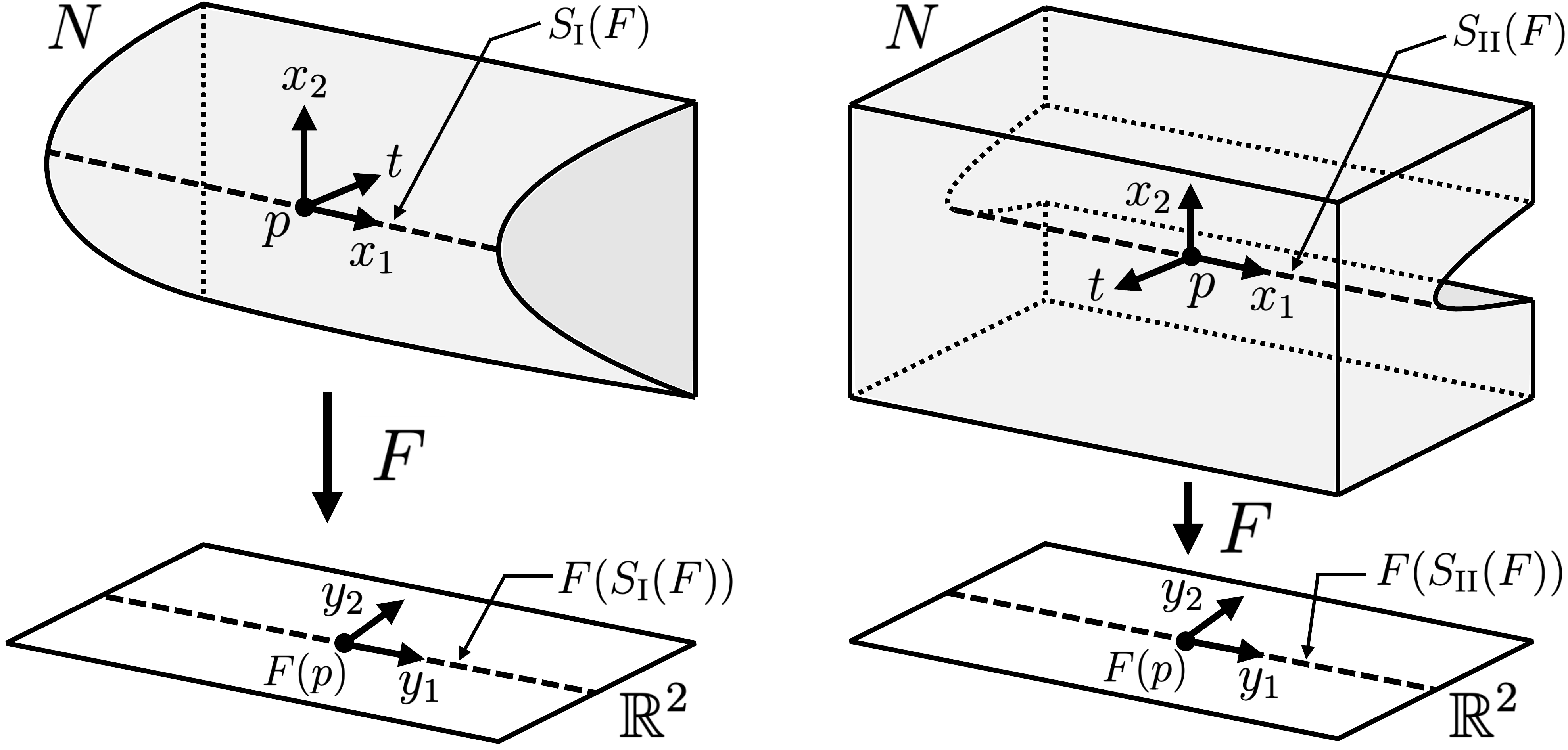}
\vspace{-4pt}
\caption{The figure on the left- (resp.\ right-) hand side represents $F$ around $p\in S_{\I}(F)$ (resp.\ $p\in S_{\II}(F)$).
}
\end{figure}

We next give a proposition that is used frequently in this paper. 
That has been proved by Shibata \cite[Proposition 1]{Shi}.

\begin{prop}
Let $N$ be a compact $3$-dimensional manifold with non-empty boundary, and $F\colon N\to\mathbb{R}^{2}$ be a submersion, which is a stable fold map on $\partial N$. 
Let $p\in\partial N$. 
Then, there exist local coordinates $(t, x_{1},x_{2})$ of $N$ around $p$ and $(y_{1},y_{2})$ of $\mathbb{R}^{2}$ around $F(p)$ such that $F$ satisfies one of the following around $p$:
\begin{equation*}
F\colon (t,x_{1},x_{2})\mapsto (y_{1},y_{2})=
\begin{cases}
(x_{1}, x_{2}), & p:\text{regular point of }F|_{\partial N},\\
(x_{1}, x_{2}^{2}\pm t), & p:\text{singular point of }F|_{\partial N}, 
\end{cases}
\end{equation*}
where $t=0$ and $t>0$ correspond to $\partial N$ and $\Int N$, respectively. 
\end{prop}

In \textrm{Figure\;1}, we draw the submersion $F$ around the singular point of $F|_{\partial N}$. 
We denote the subset of $S(F|_{\partial N})$ consisting of those points for which $F$ is of the form $(x_{1},x_{2}^{2}+t)$ (resp.\ $(x_{1},x_{2}^{2}-t)$) by $S_{\I}(F)$ (resp.\ $S_{\II}(F)$).

\begin{rem}
Let $M$ be a closed surface. 
\textrm{Proposition\;2.5} holds similarly for a submersion $g\colon M\times[0,1)\to\mathbb{R}^{2}$ which is a stable fold map of $M\times\{0\}$. 
The subset of $S(g|_{M\times\{0\}})$ consisting of the points where $g$ is of the form $(x_{1},x_{2}^{2}+t)$ (resp.\ $(x_{1},x_{2}^{2}-t)$) is denoted by $S_{\I}(g)$ (resp.\ $S_{\II}(g)$).
\end{rem}

\begin{rem}
Following Shibata\;\cite{Shi}, we refer the points contained in $S_{\I}(F)$ (resp.\ $S_{\II}(F)$) as type\;$\I$ (resp.\ type\;$\II$).
However, in other works such as Saeki--Yamamoto\;\cite{SY} and Yamamoto\;\cite{Y}, points of type\;$\I$ (resp.\ type\;$\II$) are referred to as \emph{boundary definite fold points} (resp.\ \emph{boundary indefinite fold points}). 
\end{rem}

In the rest of this subsection, we prepare some terms to define a pairing map which appears in \textrm{Theorem\;4.1}. 
Let $M$ be a closed oriented surface and $f\colon M\to\mathbb{R}^2$ be a horizontal stable fold map.
We will define the signed graph of $f$, which was originally introduced in \cite{HJF}.
For the definition of the graph, it is necessary to fix an orientation on $M$. 
Accordingly, we assume a fixed orientation on each closed orientable surface throughout this section.

\begin{defi}[\textbf{Signed graph}]
We define the vertices, edges, and labels for the \emph{signed graph} $G_{f}$ of $f$ as follows. 
The vertices of the graph correspond bijectively to the connected components of $M\setminus S(f)$.
The edges correspond bijectively to the components of $S(f)$, and each edge connects the two vertices associated with the adjacent components of $M\setminus S(f)$.
Each vertex is labeled with ``$+$" (resp.\ ``$-$") if the restriction of $f$ to the corresponding component of $M\setminus S(f)$ is orientation preserving (resp.\ reversing).
\end{defi}

\begin{rem}
In \cite{HJF}, they define the \emph{weighted graphs} for stable maps from surfaces to the plane. 
It is defined to post the genera of the regular components to the corresponding vertices. 
In this paper, we use ``signed graphs" since we do not need the weights. 
\end{rem}

We define some symbols related to the graph $G_f$.
Denote by $V_{f}$ the vertex set of $G_f$. 
The subset $V^{+}_{f}$ (resp.\ $V^{-}_{f}$) of $V_{f}$ consists of all vertices labeled with ``$+$" (resp.\ ``$-$").
Note that $V_{f}$ coincides with the disjoint union of $V_{f}^{+}$ and $V_{f}^{-}$.

\begin{rem}
As noted in \cite{HJF2}, $G_f$ is a bipartite graph with the parts $V_f^+$ and $V_f^-$.
Note that $G_f$ is not a simple graph in general; however, it has no loops because it is bipartite. 
\end{rem}

An element $(v,w)$ of $V_f^+\times V_f^-$ is identified with an edge of the graph with $V_f$ as the vertex set incident to the vertices $v \in V_f^+$ and $w \in V_f^-$.
Then, an element of $\mathcal{P}(V^{+}_{f}\times V^{-}_{f})$ is identified with a simple bipartite graphs with $V_f^+\sqcup V_f^-$ as the vertex set, where the symbol $\mathcal{P}$ represents the power set of the relevant set.
Note that $G_f$ is contained in $\mathcal{P}(V^{+}_{f}\times V^{-}_{f})$.

\begin{defi}[\textbf{Partial matching}]
A bipartite graph in which each vertex has degree $0$ or $1$ is called a \emph{partial matching}.
In particular, denote by $\mathcal{M}_f$ the subset of $\mathcal{P}(V^{+}_{f}\times V^{-}_{f})$ consisting of all partial matchings whose vertex set is $V_{f}^{+}\sqcup V_{f}^{-}$. 
\end{defi}

We define the map $\gamma_{f}\colon\mathcal{R}_{f}\to\mathcal{P}(V_f)$ as follows, where 
\begin{equation*}
\mathcal{R}_{f}=\{R~|~R\text{ is a component of }\mathbb{R}^2\setminus f(S(f))\}. 
\end{equation*}
For each $R\in\mathcal{R}_{f}$, the subset $\gamma_{f}(R)$ of $V_f$ consists of all elements corresponding to the connected components of $M\setminus S(f)$ that intersect with $f^{-1}(R)$. 
Furthermore, distinct two elements $R$ and $R'$ of $\mathcal{R}_{f}$ are said to be \emph{adjacent} if $\overline{R}\cap\overline{R'}$ is not empty and contains an arc.  
Then, we have the following lemma.

\begin{lem}\label{lemma1}
For any two adjacent and distinct elements $R$ and $R'$ of $\mathcal{R}_{f}$, either 
\begin{equation*}
\gamma_{f}(R)\subset\gamma_{f}(R')\text{ or }\gamma_{f}(R')\subset\gamma_{f}(R)
\end{equation*}
holds.
In particular, if $\gamma_{f}(R')\subset\gamma_{f}(R)$, then we have $\sharp(\gamma_{f}(R)\setminus\gamma_{f}(R'))=2$, where one element is contained in $V_f^+$ and the other is contained in $V_f^-$.
Here, $\sharp$ denotes the cardinality of the relevant set.
\end{lem}

\begin{proof}
By the definition of a fold map, the local behavior of $f$ around the singular points is determined.
Among the components of $M\setminus S(f)$, at most two are relevant to the change between $\gamma_f(R)$ and $\gamma_f(R')$, and this change is associated with the fold points corresponding to the adjacency of $R$ and $R'$. 
In particular, since $f$ is horizontal, the change is exactly two elements. 
Furthermore, by following the local behavior around the singular points of $f$, these two elements are contained in $V_f^+$ and $V_f^-$, respectively. 
\end{proof}

Under the above preparations, let us define a pairing map.
Let $g\colon M\times[0,1)\to\mathbb{R}^2$ be a submersion such that $f=g|_{M\times\{0\}}$ is a horizontal stable fold map.
In the following context, we regard $f(S(f))$ as a disjoint union of finitely many finite (multi-) graphs and circles, where the vertices correspond to the double points.

\begin{defi}[\textbf{Pairing map}]
A map $\delta\colon\mathcal{R}_{f}\to\mathcal{M}_f$ is called a \emph{pairing map} for $g$ if it satisfies the following conditions. 
\vspace{-1pt}
\begin{enumerate}
\setlength{\parskip}{0mm} 
\setlength{\itemsep}{0cm}

\item[(1)] 
For each $R\in\mathcal{R}_{f}$ and $(v,w)\in\delta(R)$, we have $v,w\in\gamma_f(R)$.
\vspace{1pt}

\item[(2)] 
For each $R\in\mathcal{R}_{f}$, $v\in\gamma_{f}(R)$ appear as an endpoint of exactly one edge in $\delta(R)$. 
\vspace{1pt}

\item[(3)] 
For each edge or circle component $e\subset f(S(f))$, let $R$ and $R'$ be components of $\mathbb{R}^2\setminus f(S(f))$ satisfying $\overline{R}\cap\overline{R'}\supset e$. 
Assume that we have $\gamma_{f}(R)\setminus\gamma_{f}(R')=\{v,w\}$, where $v\in V_f^+$ and $w\in V_f^-$.
Then, either of the following conditions holds. 
\vspace{1pt}

\begin{enumerate}

\item[(3--1)] 
Suppose that we have $f^{-1}(\Int e)\cap S_{\I}(g)\neq\emptyset$, where $\Int e$ denotes the open interval obtained by removing the endpoints of $e$ (resp.\ itself) if $e$ is an arc (resp.\ a circle).
Then, we have
\begin{align*}
&\delta(R)=\delta(R')\sqcup\{(v,w)\}.
\end{align*}

\item[(3--2)] 
Suppose that we have $f^{-1}(\Int e)\cap S_{\II}(g)\neq\emptyset$.
Then, there exists $(v',w')\in\delta(R')$ such that either of the following conditions holds. 
\begin{align*}
&\delta(R)=(\delta(R')\setminus\{(v,w)\})\sqcup\{(v',w),(v,w')\}, \text{or}\\
&\delta(R)=\delta(R')\sqcup\{(v,w)\},
\end{align*}
where, starting from any component of $\mathbb{R}^2\setminus f(M)$ and reaching a component of $\mathcal{R}_f$, the number of left-to-right changes of the second equation minus the number of right-to-left changes is non-negative. 
\end{enumerate}

\begin{figure}[t]
\centering
\includegraphics[width=50mm]{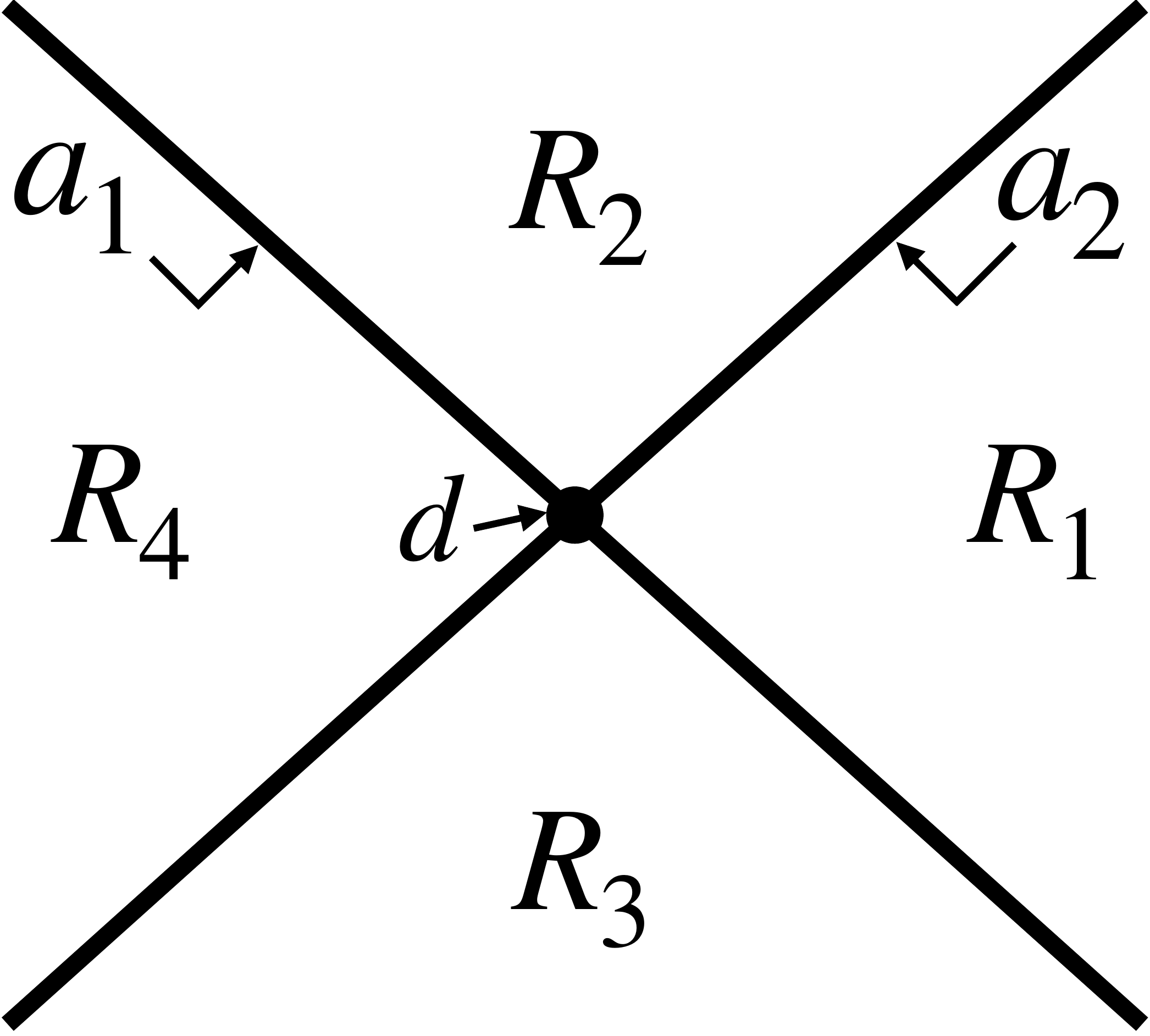}
\caption{
We depicts the regions $R_1$, $R_2$, $R_3$, and $R_4$ in $\mathcal{R}_f$ around the vertex $d$ of $f(S(f))$. 
The arcs $a_1$ and $a_2$ pass through $d$. 
}
\end{figure}

\item[(4)] 
For each vertex $d\in f(S(f))$, let $R_1,R_2,R_3$, and $R_4$ represent elements of $\mathcal{R}_f$ around $d$ as depicted in \textrm{Figure~2} such that
\begin{equation*}
\sharp\gamma_{f}(R_1)>\sharp\gamma_{f}(R_2)=\sharp\gamma_{f}(R_3)>\sharp\gamma_{f}(R_4).
\end{equation*}
The arcs $a_1$ and $a_2$ pass through $d$. 
Then, either of the following conditions holds, where the vertices $v_i\in V_f^+$ and $w_i\in V_f^-$ are all distinct for $i=1,2,3,4$. 

\vspace{1pt}
\begin{enumerate}
\item[(4--1)] Suppose that both $a_1$ and $a_2$ are contained in $f(S_{\I}(g))$. 
We have
\begin{align*}
&~~~~~~~~~~~~~\delta(R_1)=\delta(R_2)\sqcup\{(v_1,w_1)\}=\delta(R_3)\sqcup\{(v_2,w_2)\}=\delta(R_4)\sqcup\{(v_1,w_1),(v_2,w_2)\}.
\end{align*}

\item[(4--2)] Suppose that $a_1$ and $a_2$ are contained in $f(S_{\I}(g))$ and $f(S_{\II}(g))$, respectively.
Then, either of the conditions holds. 
\begin{enumerate}
\item[(4--2--1)]
Suppose that $a_2$ holds the first equation in (3--2). 
We have
\begin{align*}
&\delta(R_1)=\delta(R_3)\sqcup\{(v_1,w_1)\},\\
&\delta(R_2)=\delta(R_4)\sqcup\{(v_1,w_1)\},\\
&\delta(R_1)=(\delta(R_2)\setminus\{(v_2,w_2)\})\sqcup\{(v_2,w_3),(v_3,w_2)\},\\
&\delta(R_3)=(\delta(R_4)\setminus\{(v_2,w_2)\})\sqcup\{(v_2,w_3),(v_3,w_2)\}.
\end{align*}
\item[(4--2--2)]
Suppose that $a_2$ holds the second equation in (3--2).
We have
\begin{align*}
&\delta(R_1)=\delta(R_3)\sqcup\{(v_1,w_1)\},\\
&\delta(R_2)=\delta(R_4)\sqcup\{(v_1,w_1)\},\\
&\delta(R_1)=\delta(R_2)\sqcup\{(v_2,w_2)\},\\
&\delta(R_3)=\delta(R_4)\sqcup\{(v_2,w_2)\}.
\end{align*}
\end{enumerate}

For the case that $a_1$ and $a_2$ are contained in $f(S_{\II}(g))$ and $f(S_{\I}(g))$, respectively, we also consider similar condition.

\item[(4--3)] 
Suppose that both $a_1$ and $a_2$ are contained in $f(S_{\II}(g))$.
Then, either of the following conditions holds. 
\begin{enumerate}
\item[(4--3--1)]
Suppose that $a_1$ and $a_2$ hold the first condition of (3--2). 
We have 
\begin{align*}
& \delta(R_1)=(\delta(R_3)\setminus\{(v_1,w_1)\})\sqcup\{(v_1,w_2),(v_2,w_1)\},\\
& \delta(R_2)=(\delta(R_4)\setminus\{(v_1,w_1)\})\sqcup\{(v_1,w_2),(v_2,w_1)\},\\
& \delta(R_1)=(\delta(R_2)\setminus\{(v_3,w_3)\})\sqcup\{(v_3,w_4),(v_4,w_3)\},\\
& \delta(R_3)=(\delta(R_4)\setminus\{(v_3,w_3)\})\sqcup\{(v_3,w_4),(v_4,w_3)\},
\end{align*}
or 
\begin{align*}
& \delta(R_1)=(\delta(R_3)\setminus\{(v_1,w_2)\})\sqcup\{(v_1,w_3),(v_2,w_2)\},\\
& \delta(R_2)=(\delta(R_4)\setminus\{(v_3,w_2)\})\sqcup\{(v_3,w_3),(v_2,w_2)\},\\
& \delta(R_1)=(\delta(R_2)\setminus\{(v_3,w_3)\})\sqcup\{(v_3,w_1),(v_1,w_3)\},\\
& \delta(R_3)=(\delta(R_4)\setminus\{(v_3,w_2)\})\sqcup\{(v_3,w_1),(v_1,w_2)\}.
\end{align*}
\item[(4--3--2)] 
Suppose that $a_1$ and $a_2$ hold the first and second conditions in (3--2), respectively.
Then, we have
\begin{align*}
& \delta(R_1)=(\delta(R_3)\setminus\{(v_1,w_1)\})\sqcup\{(v_2,w_1),(v_1,w_2)\},\\
& \delta(R_2)=(\delta(R_4)\setminus\{(v_1,w_1)\})\sqcup\{(v_2,w_1),(v_1,w_2)\},\\
& \delta(R_1)=\delta(R_2)\sqcup\{(v_3,w_3)\},\\
& \delta(R_3)=\delta(R_4)\sqcup\{(v_3,w_3)\}.
\end{align*}
A similar equation holds when $a_1$ and $a_2$ satisfy the opposite conditions. 
\item[(4--3--3)]
Suppose that $a_1$ and $a_2$ hold the second condition in (3--2). 
We have
\begin{align*}
&\delta(R_1)=\delta(R_3)\sqcup\{(v_1,w_1)\},\\
&\delta(R_2)=\delta(R_4)\sqcup\{(v_1,w_1)\},\\
&\delta(R_1)=\delta(R_2)\sqcup\{(v_2,w_2)\},\\
&\delta(R_3)=\delta(R_4)\sqcup\{(v_2,w_2)\}.
\end{align*}
\end{enumerate}
\end{enumerate}
\end{enumerate}
\end{defi}

\section{Examples}
In this section, we present two examples, \textrm{Example\;3.1} and \textrm{Example\;3.2}, of submersions of the product of closed oriented surfaces and $[0,1)$ into $\mathbb{R}^2$ and their pairing maps. 
By the existence of these pairing maps, each of the submersions admits a non-singular as will follow from \textrm{Theorem\;4.1} stated later.
We also provide an example of a map that does not admit a non-singular extension in \textrm{Example\;3.3}.

\begin{example}
In \textrm{Figure\;3}, the left-hand side shows the submersion $g_1\colon S^2\times[0,1)\to\mathbb{R}^2$ that is defined the composition of an embedding of $S^2\times[0,1)$ into $\mathbb{R}^3$ and the standard projection from $\mathbb{R}^3$ to $\mathbb{R}^2$, while the right-hand side shows the signed graph $G_{f_1}$ of the restriction $f_1=g_1|_{M\times\{0\}}$ of $g_1$ to the boundary $M\times\{0\}$.
For the map $g_1$, a pairing map $\delta_1$ is introduced as shown in \textrm{Figure\;4}. 
In the figure, the map $\delta_1$ corresponds each element of $\mathcal{R}_{f_1}$ to the graphs in $\mathcal{M}_{f_1}$. 
\end{example}

\begin{example}
In \textrm{Figure\;5}, the left-hand side shows the submersion $g_2\colon (S^2\sqcup S^2)\times[0,1)\to\mathbb{R}^2$ that is also define as the composition of an embedding from $(S^2\sqcup S^2)\times[0,1)$ to $\mathbb{R}^3$ and the standard projection of $\mathbb{R}^3$ into $\mathbb{R}^2$, and the right-hand figure shows the signed graph $G_{f_2}$ corresponding to $f_2=g_2|_{(S^2\sqcup S^2)\times\{0\}}$.
The map $g_2$ has a pairing map $\delta_2$ shown in \textrm{Figure\;6}. 
In the figure, the map $\delta_2$ corresponds each element of $\mathcal{R}_{f_2}$ to the graphs in $\mathcal{M}_{f_2}$. 
\end{example}

\begin{example}
In \textrm{Figure\;7}, we depict a submersion $g_3\colon S^2\times[0,1)\to\mathbb{R}^2$.
This map is an example that does not admit a non-singular extension. 
\end{example}

\begin{figure}[t]
\centering
\includegraphics[width=80mm]{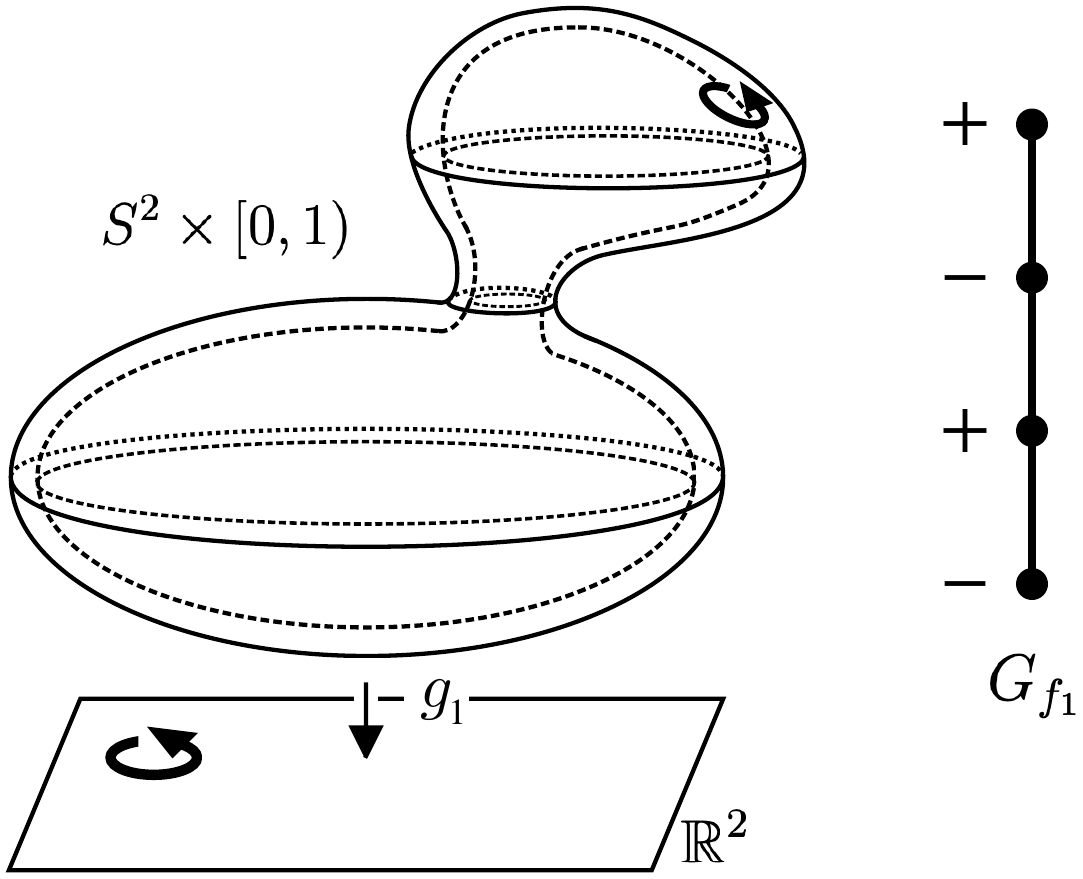}
\caption{
The submersion $g_1$ and the signed graph $G_{f_1}$. 
}
\end{figure}

\begin{figure}[t]
\centering
\includegraphics[width=80mm]{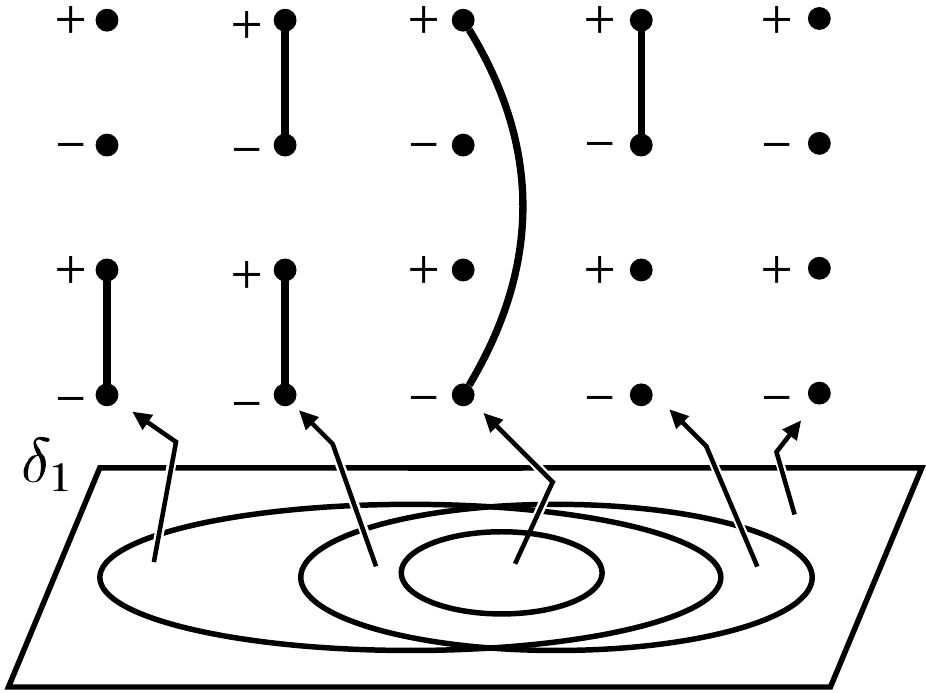}
\caption{
The pairing map $\delta_1$ from $\mathcal{R}_{f_1}$ to $\mathcal{M}_{f_1}$. 
}
\end{figure}

\begin{figure}[t]
\centering
\includegraphics[width=77mm]{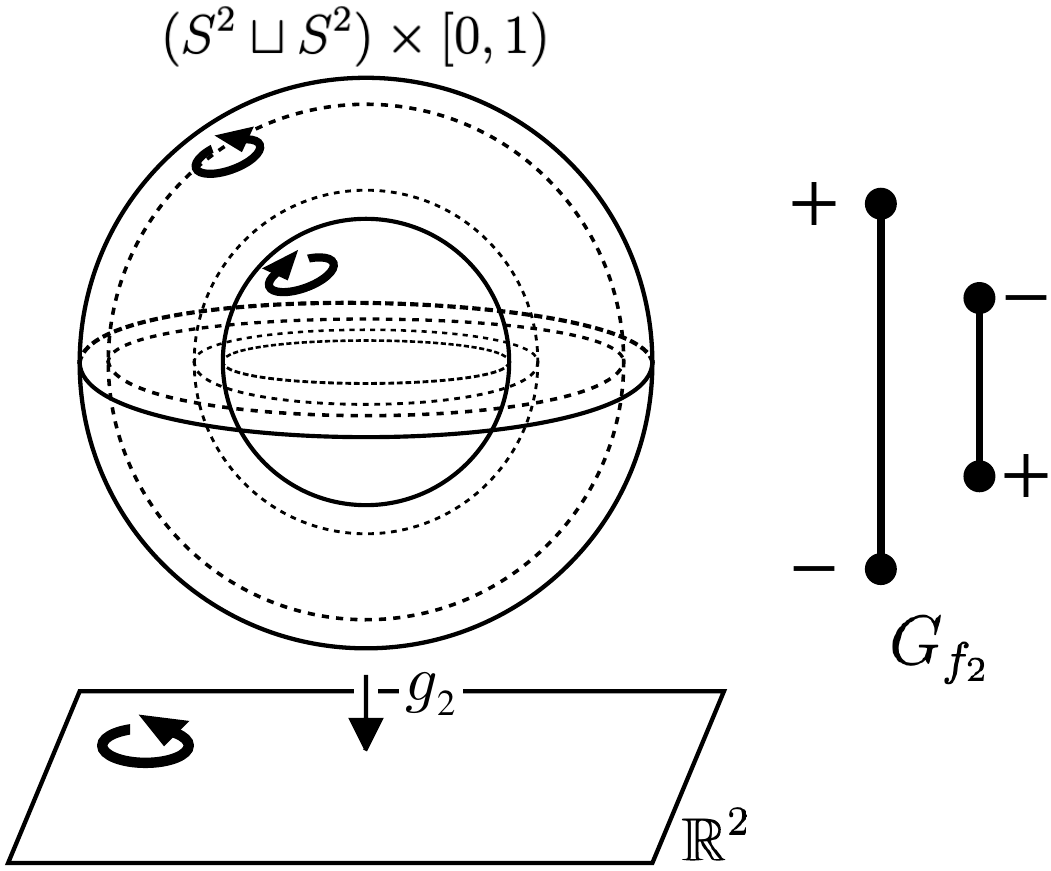}
\caption{
The submersion $g_2$ and the signed graph $G_{f_2}$. 
}
\end{figure}

\begin{figure}[t]
\centering
\includegraphics[width=80mm]{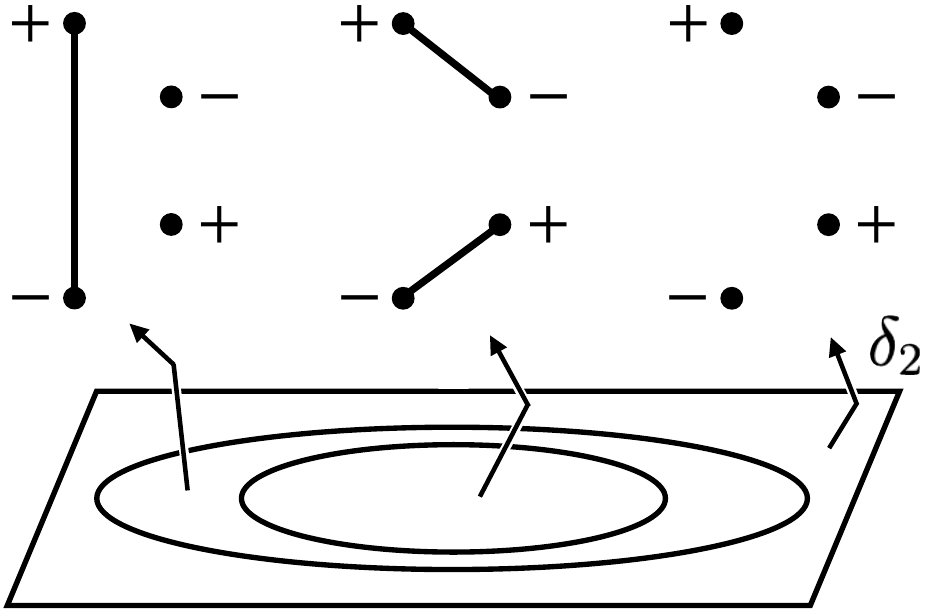}
\caption{
The pairing map $\delta_2$ from $\mathcal{R}_{f_2}$ to $\mathcal{M}_{f_2}$. 
}
\end{figure}

\begin{figure}[t]
\centering
\includegraphics[width=60mm]{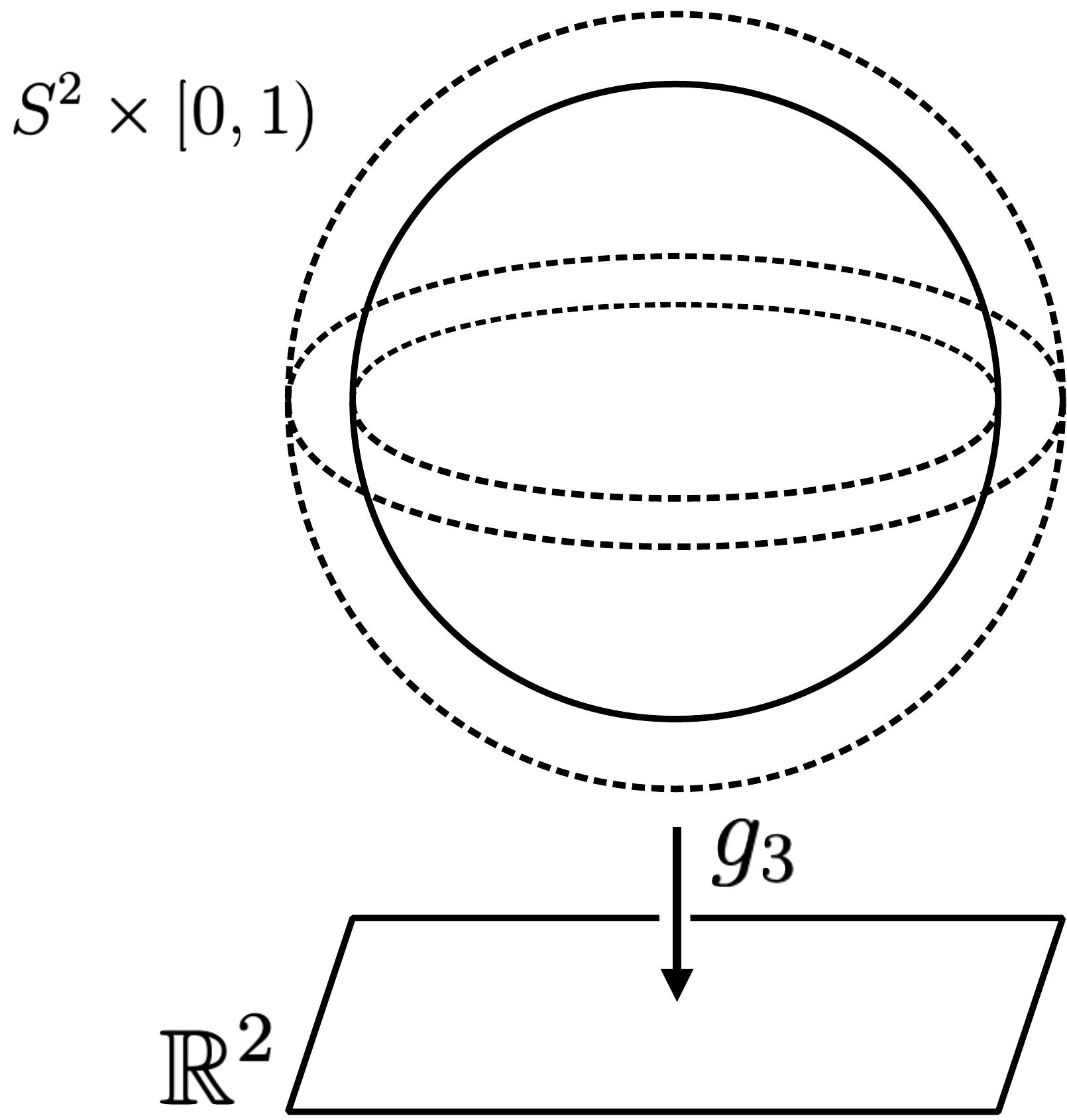}
\caption{
The submersion $g_3$. 
}
\end{figure}

\section{Main theorem and its proof}

In this section, we give our main theorem and its proof. 

\begin{thm}
Let $M$ be a closed oriented surface and let $g\colon M\times[0,1)\to\mathbb{R}^{2}$ be a submersion that is a horizontal stable fold map $f=g|_{M\times\{0\}}$ on the boundary $M\times\{0\}$. 
Then, the following statements are equivalent. 
\begin{enumerate}
\setlength{\leftskip}{0pt}
\item[(1)] There exist a compact oriented $3$-dimensional manifold $N$ such that $\partial N=M$, and a non-singular extension $F\colon N\to\mathbb{R}^{2}$ of $g$.
\vspace{3pt}
\item[(2)] There exists a pairing map $\delta\colon\mathcal{R}_f\to\mathcal{M}_f$ for $g$.
\vspace{1pt}
\end{enumerate}
\end{thm}

\proof
First, we define a map $\delta_F\colon\mathcal{R}_f\to\mathcal{M}_f$ for a non-singular extension $F\colon N\to\mathbb{R}^2$ of $g$.
Fix any region $R\in\mathcal{R}_f$ and choose a point $r\in R$. 
Since $r$ is a regular value of $f$, the preimage $F^{-1}(r)$ is a $1$-dimensional manifold, namely it is composed of arcs and circles. 
Each arc has two endpoints that lie in connected components of $M\setminus S(f)$. 
Besides, they correspond to two distinct vertices of $G_f$ with opposite labels because $N$ is oriented.  
Then, define $\delta_F(R)$ to be the set of all elements $(v,w)\in V_f^+\times V_f^-$ corresponding to an arc component of $F^{-1}(r)$. 
One checks immediately that this set does not depend on the choice of $r\in R$ because the two preimages of any two points in $R$ are homeomorphic. 
Moreover, regard $\delta_F(R)$ as a graph with vertex set $V_f=V_f^{+}\sqcup V_f^-$, we see that it is a partial matching, hence $\delta_F(R)\in\mathcal{M}_f$. 
This completes the definition of $\delta_F$ and verification of its well-definedness. 
\vspace{2pt}

We now verify that $\delta_F$ satisfies all the conditions of \textrm{Definition\;2.13} and therefore is a pairing map for $g$. 
First, let $R\in\mathcal{R}_f$ and $(v,w)\in\delta_F(R)$. 
By the definition of $\delta_F(R)$, the vertices $v$ and $w$ correspond to the endpoints of an arc component of $F^{-1}(r)$ for any $r\in R$. 
Since those endpoints lie in $M\setminus S(f)$, the definition of $\gamma_f(R)$ implies $v,w\in\gamma_f(R)$. 
Hence, $\delta_F$ satisfies condition (1). 
Let $v\in\gamma_f(R)$. 
By the definition of $\delta_F$, there is at least one edge in $\delta_F(R)$ incident to $v$. 
Moreover, Ehresmann's fibration theorem\;\cite{E} guarantees that exactly one such edge exists. 
Thus, condition (2) is satisfied. 
Let $R,R'\in\mathcal{R}_f$ be adjacent regions. 
Since $f$ is a horizontal stable fold map, \textrm{Lemma\;2.12} shows that $\gamma_f(R)$ and $\gamma_f(R')$ differ in exactly two vertices. 
Consequently, $\delta_F(R)$ and $\delta_F(R')$ differ by one edge joining those two vertices. 
Consider the types, $\I$ or $\II$, the singular value set of $f$ to which $R$ and $R'$ adjacent. 
\textrm{Proposition\;2.5} lists all possible local forms of $F$ around points on $\partial N$. 
According to this proposition, we immediately confirm that the difference between $\delta_F(R)$ and $\delta_F(R')$ corresponds to (3--1) or (3--2) of condition (3). 
The number defined in (3--2) is non-negative, since it is the number of $S^1$-components in the inverse image under $F$ of a point in the relevant element of $\mathcal{R}_f$. 
Hence, condition (3) is satisfied. 
A similar argument to condition (3) in a neighborhood of double points of $f|_{S(f)}$ shows that $\delta_F$ holds one of the sub-conditions in (4). 
This completes the verifications that $\delta_F$ is a pairing map for $g$. 
\vspace{2pt}

We now prove the converse of the theorem. 
Let $\delta\colon\mathcal{R}_f\to\mathcal{M}_f$ be a pairing map for $g$. 
To construct a non-singular extension of $g$, we first carry out some preparations for $\mathbb{R}^2$. 
We begin by decomposing $\mathbb{R}^2$ as follows. 
\begin{equation*}
\mathbb{R}^2=K\cup\bigsqcup_{R\in\mathcal{R}_f}L_{R}, 
\end{equation*}
where $K$ is a regular neighborhood of $f(S(f))$ in $\mathbb{R}^2$, and $L_R$ is the closure of the connected component of $\mathbb{R}^2\setminus K$ contained in $R$. 
We further decompose $K$ as
\begin{equation*}
K=\bigsqcup_{i=1}^{a}D_i \cup\bigsqcup_{i=1}^{b_1}T_i \cup\bigsqcup_{i=1}^{b_2}A_i,  
\end{equation*}
where each disk $D_i$ contains exactly one vertex of $f(S(f))$; each rectangle $T_i$ contains an arc that is a component of 
$
f(S(f))\cap\big(
\overline{\mathbb{R}^2\setminus\bigsqcup_{i=1}^{a}D_i}
\big);
$
and each annulus $A_i$ contains a circle component of $f(S(f))$.
Here, the integers $a$, $b_1$, and $b_2$ denote the numbers of vertices, edges, and circle components of $f(S(f))$, respectively, and they are finite because $M$ is compact. 
\vspace{2pt}

We next introduce a collection of sets and maps from them to the plane that will serve as the building block for constructing a non-singular extension of $g$. 
This construction follows the decomposition of $\mathbb{R}^2$ given above. 
\vspace{2pt}

Let $R\in\mathcal{R}_f$. 
Then, we define 
\begin{equation*}
N_R=L_R\times\Bigg(
\bigsqcup_{(v,w)\in\delta(R)}J_{(v,w)}
\sqcup 
\bigsqcup_{j=1}^{c}S_j^1
\Bigg), 
\end{equation*}
where each $J_{(v,w)}$ is an arc and each $S_j^1$ is a circle. 
The number $c$ is defined as the number of times the value of $\delta$ changes from the left-hand side to the right-hand side of the second equation in condition (3--2) as one moves from a component of $\mathbb{R}^2\setminus f(M)$ to $R$, minus the number of times it change from the right-hand side to the left-hand side. 
This is non-negative by the condition in (3--2). 
Furthermore, we define a map $F_R\colon N_R\to\mathbb{R}^2$ by $(x,y)\mapsto x$, where $x\in L_R$ and $y\in\bigsqcup_{(v,w)\in\delta(R)}J_{(v,w)}\sqcup \bigsqcup_{j=1}^{c}S_j^1$. 
\vspace{2pt}

\begin{figure}[t]
\centering
\includegraphics[width=140mm]{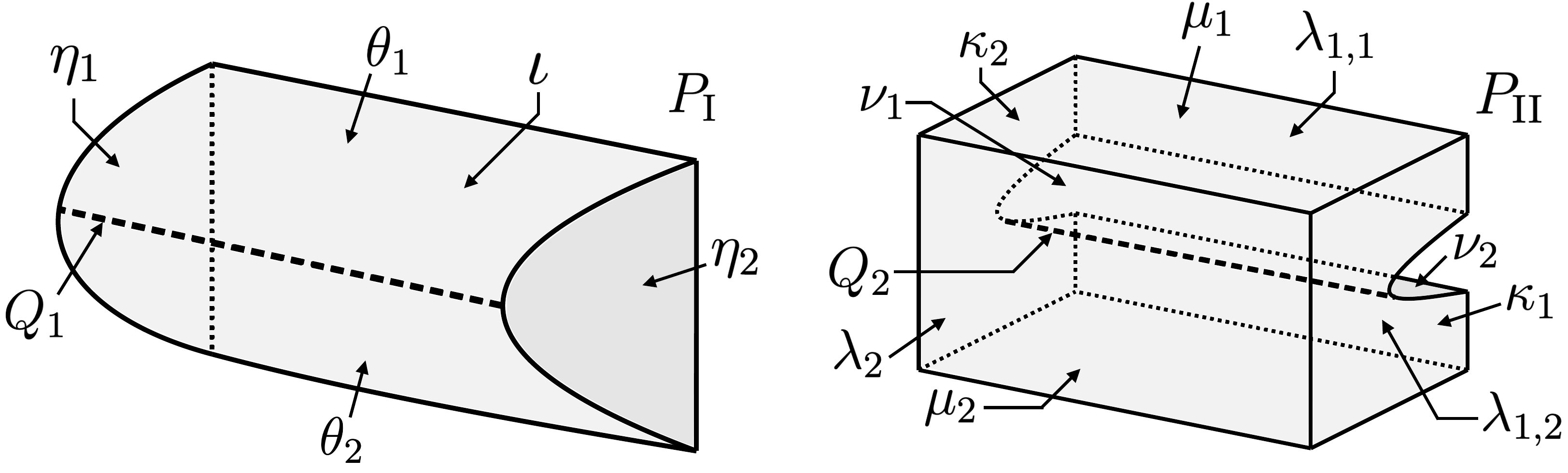}
\caption{
The part $P_{\I}$ and $P_{\II}$. 
The part $P_{\I}$ has five regions, denoted by $\eta_1$, $\eta_2$, $\theta_1$, $\theta_2$, and $\iota$, while $P_{\II}$ has nine regions, denoted by $\kappa_1$, $\kappa_2$, $\lambda_{1,1}$, $\lambda_{1,2}$, $\lambda_2$, $\mu_1$, $\mu_2$, $\nu_1$, and $\nu_2$.
Additionally, they have lines $Q_1$ and $Q_2$, respectively. 
}
\end{figure}

Let $e$ be the edge of $f(S(f))$ intersecting with the rectangle $T_i$. 
Assume that $\delta$ satisfies condition (3--1) of \textrm{Definition\;2.13}. 
In other words, for two regions $R,R'\in\mathcal{R}_f$ adjacent along $e$, we have
\begin{equation*}
\delta(R)=\delta(R')\sqcup\{(v,w)\},
\end{equation*}
In this case, we define 
\begin{equation*}
N_{T_i}=
\Bigg(
T_i\times
\Bigg(\bigsqcup_{(v',w')\in\delta(R')}J_{(v',w')}
\sqcup\bigsqcup_{j=1}^{c}S_j^1
\Bigg)
\Bigg)
\sqcup
P_{\I}, 
\end{equation*}
where each $J_{(v',w')}$ is an arc and each $S_j^1$ is a circle. 
Here, the part $P_{\I}$ is depicted at the left-hand side of \textrm{Figure\;8}. 
The integer $c\geq 0$ is defined for $R'$ as in the previous case. 
Next, we define a map $F_{T_i}\colon N_{T_i}\to\mathbb{R}^2$ as follows. 
On the product part $T_i\times\big(\bigsqcup_{(v',w')\in\delta(R')}J_{(v',w')}\sqcup\bigsqcup_{j=1}^{c}S_j^1\big)$, we define $F_{T_i}$ as $(x,y)\mapsto x$, where $x\in T_i$ and $y\in\bigsqcup_{(v',w')\in\delta(R')}J_{(v',w')}\sqcup\bigsqcup_{j=1}^{c}S_j^1$. 
On the part $P_{\I}$, we define $F_{T_i}$ as depicted in \textrm{Figure\;9}. 
The image of $P_{\I}$ by $F_{T_i}$ is $R\cap T_i$. 
Moreover, the image of the line $Q_1$ is $e\cap T_i$, while the regions $\eta_1$, $\eta_2$, and $\iota$ are each mapped to distinct edges of the rectangle $R\cap T_i$. 
Thus, we obtain a map $F_i$ from $N_{T_i}$ to $\mathbb{R}^2$. 
\vspace{2pt}

\begin{figure}[t]
\centering
\includegraphics[width=115mm]{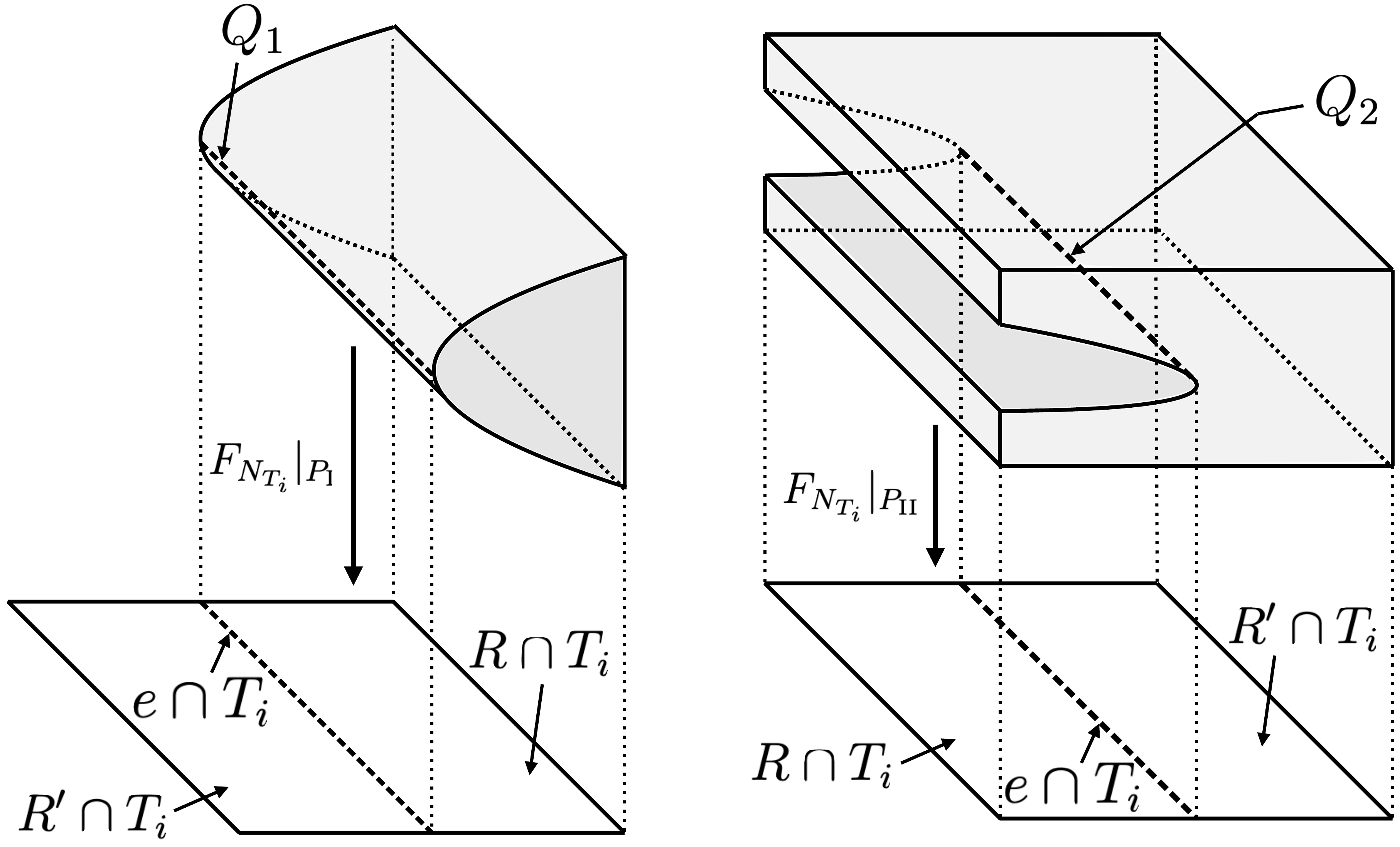}
\caption{
The map $F_{N_{T_i}}|_{P_{\I}}$ and $F_{N_{T_i}}|_{P_{\II}}$. 
}
\end{figure}

Assume that $\delta$ satisfies the first equation in (3--2) of \textrm{Definition\;2.13} along the edge $e$ of $f(S(f))$. 
That is, for two regions $R,R'\in\mathcal{R}_f$ adjacent along $e$, we have 
\begin{equation*}
\delta(R)=(\delta(R')\setminus\{(v,w)\})\sqcup\{(v',w),(v,w')\}.
\end{equation*}
In this case, we define
\begin{equation*}
N_{T_i}=
\Bigg(
T_i\times
\Bigg(
\bigsqcup_{(v'',w'')\in\delta(R')\setminus \{(v,w)\}}J_{(v'',w'')}
\sqcup
\bigsqcup_{j=1}^{c}S_j^1
\Bigg)
\Bigg)
\sqcup
P_{\II},
\end{equation*}
where each $J_{(v'',w'')}$ and each $S_j^1$ are defined as before. 
The integer $c\geq 0$ is defined for the region $R$ in the same way as in the previous cases. 
The part $P_{\II}$ is depicted on the right-hand side in \textrm{Figure\;8}. 
We now define a map $F_{T_i}\colon N_{T_i}\to\mathbb{R}^2$. 
On the product part, it is defined by $(x,y)\to x$, where $x\in T_i$ and $y\in\bigsqcup_{(v'',w'')\in\delta(R')\setminus \{(v,w)\}}J_{(v'',w'')}\sqcup\bigsqcup_{j=1}^{c}S_j^1$. 
On the part $P_{\II}$, it is defined as depicted in \textrm{Figure\;9}. 
The image of $P_{\II}$ by $F_{T_i}$ is $T_i$, and it is mapped compatibly with the equation in (3--2) for $\delta$. 
Here, the line $Q_2$ is mapped to $e\cap T_i$. 
Moreover, the regions $\lambda_{1,1}$, $\lambda_{1,2}$, $\lambda_2$, $\kappa_1$, and $\kappa_2$ are mapped to the edges of the rectangle $T_i$, with $\lambda_{1,1}$ and $\lambda_{1,2}$ mapped to the same edge, and $\lambda_{1,1}$, $\lambda_2$, $\kappa_1$, and $\kappa_2$ each mapped to distinct edges. 
\vspace{2pt}

Assume that $\delta$ satisfies the second equation in (3--2) of \textrm{Definition\;2.13} along the edge $e$ of $f(S(f))$.
That is, for two regions $R,R'\in\mathcal{R}_f$ adjacent along $e$, we have 
\begin{equation*}
\delta(R)=\delta(R')\sqcup\{(v,w)\}.
\end{equation*} 
In this case, we define
\begin{equation*}
N_{T_i}=
\Bigg(
T_i\times
\Bigg(
\bigsqcup_{(v',w')\in\delta(R')}J_{(v',w')}
\sqcup
\bigsqcup_{j=1}^{c}S_j^1
\Bigg)
\Bigg)
\sqcup
\tilde{P}_{\II},
\end{equation*}
where each $J_{(v',w')}$ and each $S_j^1$ are defined as in the previous cases, and the integer $c\geq 0$ is defined for the region $R'$ similarly to the earlier cases. 
Here, the part $\tilde{P}_{\II}$ is obtained by identifying the regions $\mu_1$ and $\mu_2$ in the part $P_{\II}$. 
Moreover, the map $F_{T_i}\colon N_{T_i}\to\mathbb{R}^2$ is defined as in the case of the first equation in (3--2). 
\vspace{2pt}

We next define a set $N_{A_i}$ and a map $F_{A_i}\colon N_{A_i}\to\mathbb{R}^2$ for $A_i$ containing a circle component of $f(S(f))$. 
They are similarly defined as the case for $T_i$. 
In the construction, we use the parts $P'_{\I}$, $P'_{\II}$, and $\tilde{P}'_{\II}$ instead of the parts $P_{\I}$, $P_{\II}$, and $P'_{\II}$, respectively.
These parts are obtained by attaching $\eta_1$ and $\eta_2$ of $P_{\I}$, $\kappa_1$ and $\kappa_2$ of $P_{\II}$, and $\kappa_1$ and $\kappa_2$ of $\tilde{P}_{\II}$, respectively. 
\par
\vspace{2pt}

Finally, let us consider the square $D_i$ containing the vertex $d$ of $f(S(f))$.
We will define a set $N_{D_i}$ and a map $F_{D_i}\colon N_{D_i}\to\mathbb{R}^2$. 
Here, we distinguish cases according to how the two arcs $a_1$ and $a_2$ passing through $d$ are corresponded to either $\I$ or $\II$. 
\vspace{2pt}

We begin with the case in which $a_1,a_2\subset f(S_{\I}(f))$. 
Then, $\delta$ satisfies condition (4--1) of \textrm{Definition\;2.13}. 
That is, there exists four regions $R_1,R_2,R_3,R_4\in\mathcal{R}_f$ around $d$ (see \textrm{Figure\;2}), and $\delta$ holds the following equation: 
\begin{equation*}
\delta(R_1)=\delta(R_2)\sqcup\{(v_1,w_1)\}=\delta(R_3)\sqcup\{(v_2,w_2)\}=\delta(R_4)\sqcup\{(v_1,w_1),(v_2,w_2)\}. 
\end{equation*}
In this case, we define  
\begin{equation*}
N_{D_i}=
\Bigg(
D_i\times
\Bigg(
\bigsqcup_{(v,w)\in\delta(R_4)}J_{(v,w)}
\sqcup
\bigsqcup_{j=1}^{c}S_j^1
\Bigg)
\Bigg)
\sqcup
P_{\I}
\sqcup
P_{\I},
\end{equation*}
where each $J_{(v,w)}$, each $S_j^1$, the integer $c\geq 0$, and the part $P_{\I}$ are defined as before cases.
Here, $c$ is defined for $R_4$. 
We next define a map $F_{D_i}\colon N_{D_i}\to\mathbb{R}^2$ as follows.  
We define $F_{D_i}$ by $(x,y)\mapsto x$ on the product part, where $x\in D_i$ and $y\in(\sqcup_{(v,w)\in\delta(R_4)}J_{(v,w)})
\sqcup
(\sqcup_{j=1}^{c}S_j^1)$.
On each $P_{\I}$, it is defined as shown in \textrm{Figure\;9}.
For one $P_{\I}$, it is mapped to $(R_1\cup R_2)\cap D_i$ so that $Q_1$ is sent to $a_1$, and $\eta_1$, $\eta_2$, and $\iota$ are mapped to distinct edges of the rectangle $(R_1\cup R_2)\cap D_i$.
For the other $P_{\I}$, it is mapped to $(R_1\cup R_3)\cap D_i$ so that $Q_1$ is sent to $a_2$, and $\eta_1$, $\eta_2$, and $\iota$ are mapped to distinct edges of the rectangle $(R_1\cup R_3)\cap D_i$. 
Note that $R_1$, $R_2$, $R_3$, and $R_4$ are considered to be squares in our arguments.
\vspace{2pt}

The definitions of $N_{D_i}$ and $F_{D_i}$ for (4--1) are complete by combining the definition of $N_{D_i}$ and $F_{T_i}$ in (3--1). 
Similarly, $N_{D_i}$ and $F_{D_i}$ for the first equation of (4--3--1), (4--3--2), and (4--3--3) are defined by combining the definition of $N_{D_i}$ and $F_{T_i}$ in (3).
\vspace{2pt}

\begin{figure}[t]
    \centering
    \includegraphics[width=80mm]{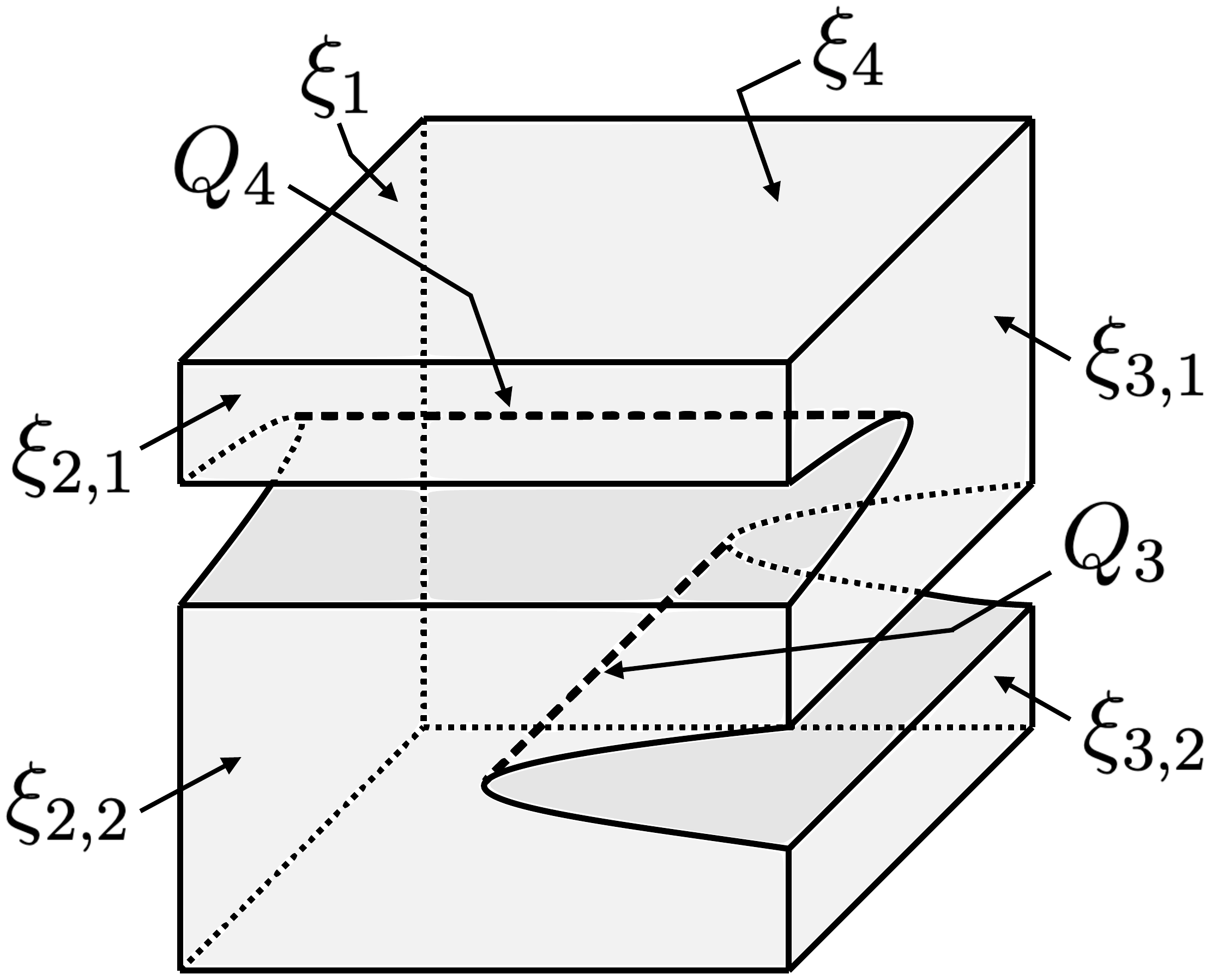}
    \caption{
The part $P_{\rm{D}}$.
It has the six regions denoted by $\xi_1$, $\xi_{2,1}$, $\xi_{2,2}$, $\xi_{3,1}$, $\xi_{3,2}$, and $\xi_4$, and the lines represented by $Q_3$ and $Q_4$.  
    }
\end{figure}

\begin{figure}[t]
    \centering
    \includegraphics[width=60mm]{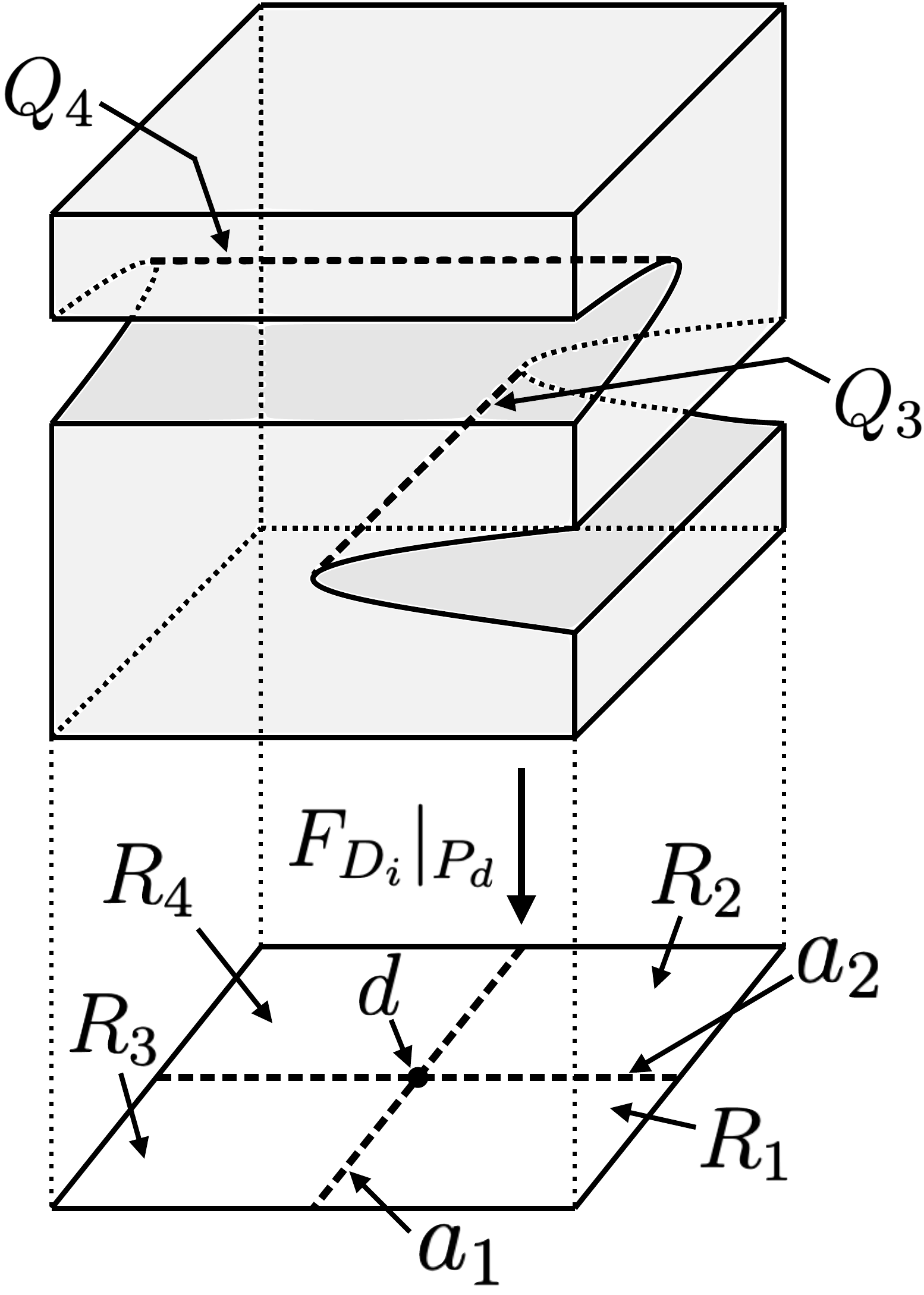}
    \caption{
The map $F_{D_i}|_{P_{\rm{D}}}$. 
    }
\end{figure}

On the other hand, when $\delta$ holds the second equation of (4--3--1) around $d$, we defined $N_{D_i}$ and $F_{D_i}$ by using the part $P_{D_i}$ as depicted in \textrm{Figure\;10}. 
That is, there exist the four regions $R_1,R_2,R_3,R_4\in\mathcal{R}_f$ around $d$ such that $\delta$ satisfies the equation 
\begin{align*}
& \delta(R_1)=(\delta(R_3)\setminus\{(v_1,w_2)\})\sqcup\{(v_1,w_3),(v_2,w_2)\},\\
& \delta(R_2)=(\delta(R_4)\setminus\{(v_3,w_2)\})\sqcup\{(v_3,w_3),(v_2,w_2)\},\\
& \delta(R_1)=(\delta(R_2)\setminus\{(v_3,w_3)\})\sqcup\{(v_3,w_1),(v_1,w_3)\},\\
& \delta(R_3)=(\delta(R_4)\setminus\{(v_3,w_2)\})\sqcup\{(v_3,w_1),(v_1,w_2)\}.
\end{align*}
In this case, we define
\begin{equation*}
N_{D_i}=
\Bigg(
D_i\times
\Bigg(
\bigsqcup_{(v,w)\in\delta(R_4)\setminus\{(v_2,w_4)\}}J_{(v,w)}
\sqcup
\bigsqcup_{j=1}^{c}S_j^1
\Bigg)
\Bigg)
\sqcup
P_{\rm{D}}, 
\end{equation*}
where each $J_{(v,w)}$, each $S_j^1$, and the integer $c\geq 0$ are defined as in the previous cases, with $c$ defined for the region $R_4$ in the same manner. 
The part $P_{D}$ is depicted as in \textrm{Figure\;10}. 
We now define the map $F_{D_i}\colon N_{D_i}\to\mathbb{R}^2$ as follows. 
On the product part, $F_{D_i}$ is defined by $(x,y)\mapsto x$, where $x\in D_i$ and $y\in\bigsqcup_{(v,w)\in\delta(R_4)\setminus\{(v_2,w_4)\}}J_{(v,w)}\sqcup\bigsqcup_{j=1}^{c}S_j^1$. 
On $P_{\rm{D}}$, it is defined as \textrm{Figure\;11} so that $P_{\rm{D}}$ is mapped to $D_i$.
Here, $Q_3$ and $Q_4$ are mapped to $a_1$ and $a_2$, respectively, and $\xi_1$, $\xi_{2,1}$, $\xi_{2,2}$, $\xi_{3,1}$, $\xi_{3,2}$, and $\xi_{4}$ are mapped to the edges of the square $D_i$. 
Note that $\xi_{2,1}$ and $\xi_{2,2}$, as well as $\xi_{3,1}$ and $\xi_{3,2}$, are mapped to the same edge of $D_i$, whereas $\xi_1$, $\xi_{2,1}$, $\xi_{3,1}$, and $\xi_4$ are each mapped to distinct edges.
\vspace{2pt}

We now assemble a $3$-dimensional manifold $N$ by gluing together the building blocks $N_R$, $N_{D_i}$, $N_{T_i}$, and $N_{A_i}$ defined above. 
To begin, set
\begin{equation*}
\tilde{N}=
\Bigg(
\bigsqcup_{R\in\mathcal{R}_f} N_{R}
\Bigg)
\sqcup 
\Bigg(
\bigsqcup_{i=1}^{a} N_{D_i}
\Bigg)
\sqcup 
\Bigg(
\bigsqcup_{i=1}^{b_1} N_{T_i}
\Bigg)
\sqcup 
\Bigg(
\bigsqcup_{i=1}^{b_2} N_{A_i}
\Bigg).
\end{equation*}
We define $N$ by identifying the points of these pieces according to the intersections $L_R\cap T_i$, $L_R\cap A_i$, $L_R\cap D_i$, and $T_i\cap D_i$. 
\vspace{2pt}

Since the identification for each case is performed similarly, we will describe the method for the intersection $L_R\cap T_i\neq\emptyset$. 
Recall that $N_R$ is the product of $L_R$ and intervals or circles, and $N_{T_i}$ is a disjoint union of the product of $T_i$ and intervals or circles, and $P_{\I}$ or $P_{\II}$. 
Here, the products of $L_R$ or $T_i$ and the boundaries of intervals correspond to the vertices in $V_g$. 
Moreover, the regions $\theta_1$ and $\theta_2$ in $P_\I$ shown in \textrm{Figure\;8} correspond to the vertices $v$ and $w$ of (3--1) in \textrm{Definition\;2.13}, respectively. 
The regions $\mu_1$, $\mu_2$, $\nu_1$, and $\nu_2$ in $P_\II$ correspond to the vertices $v$, $w$, $w'$ and $v'$ in the first equation of (3--2) in \textrm{Definition\;2.13}, respectively. 
Furthermore, the regions $\nu_1$ and $\nu_2$ of $\tilde{P}_{\II}$ correspond to $v$ and $w$ in the second equation of (3--2) in \textrm{Definition\;2.13}, respectively. 
By these correspondences, let us carry out the identifications. 
First, we consider the product parts of $N_R$ and $N_{T_i}$. 
In this case, when the boundaries of the intervals correspond to the same vertex, we identify the product parts on $\partial N_{R}$ and $\partial N_{T_i}$. 
We next consider the product parts of $N_R$ or $\partial N_{T_i}$ and the parts $P_{\I}$, $P_{\II}$, and $\tilde{P}_{\II}$. 
In this case, when the corresponding vertices are the same, we identify the product parts on $\partial N_{R}$ and $\partial N_{T_i}$ and $\iota$, $\lambda_{1,1}$, $\lambda_{1,2}$, and $\lambda_{2}$.  
Throughout this process, the identification is defined in the case where $L_R\cap T_i\neq\emptyset$. 
The identifications for the other cases are defined similarly, and the $3$-dimensional manifold $N$ is obtained from $\tilde{N}$. 
\vspace{2pt}

Recall that we defined maps $F_{N_{R}}$, $F_{T_{i}}$, $F_{A_{i}}$, and $F_{D_{i}}$ from each pieces to $\mathbb{R}^2$. 
From the definition of $N$ above, a map $F\colon N\to\mathbb{R}^2$ is naturally defined from these maps. 
\vspace{2pt}

Finally, let us verify that $F$ is a non-singular extension of $g$. 
Here, the singular point set of $F|_{\partial N}$ consists of $Q_1$, $Q_2$, $Q_3$, and $Q_4$ taken over all the relevant parts. 
Moreover, the image of this singular point set coincides with $f(S(f))$, and the orientations of them, defined as the manner in \cite{Lev}, also agree. 
Thus, by applying Levine's theorem\;\cite{Lev} to each component of $\partial N$ and $M$, we conclude that $\partial N$ is homeomorphic to $M$. 
\vspace{2pt}

We verify that $F$ coincides with $g$ on the collar neighborhood of the boundary $\partial N$. 
Let us define the embedding $i\colon M\to N$ as follows.
For $x\in M\setminus S(f)$, we consider the component of $M\setminus S(f)$ containing $x$, and define the point $y\in \partial N$ such that $F(y)=f(x)$ to be the point $i(x)$.
For $x\in S(f)$, we consider the component of $S(f)$ containing $x$, and define the point $y\in S(F|_{\partial N})$ such that $F(y)=f(x)$ to be the point $i(x)$.
Note that even when $f(x)$ is a vertex of $f(S(f))$, the point $y$ is uniquely defined by considering the component of $S(f)$ containing $x$.
Then, we naturally obtain an embedding $i\colon M\times [0,1)\to N$ from $\tilde{i}$. 
Furthermore, according to the construction of $F$ and the definition of $\tilde{i}$, it follows that $F\circ\tilde{i}=g$ holds. 
Therefore, $F$ is a non-singular extension of $g$.  
\qed

\begin{rem}
Stable maps have cusp points as their singular points in general. 
The above method is valid for such a case. 
\end{rem}

\begin{rem}
There exist submersions as in \textrm{Theorem\;4.1} that admit several non-singular extensions. 
For example, the submersion $g\colon ((S^1\times S^1)\sqcup S^2\sqcup S^2)\times [0,1)\to\mathbb{R}^2$ depicted in \textrm{Figure\;12} admits the non-singular extensions $F_1\colon ((S^1\times D^2)\setminus \Int B^3)\sqcup D^3\to\mathbb{R}^2$ and 
$F_2\colon (S^1\times D^2)\sqcup (D^3\setminus \Int B^3)\to\mathbb{R}^2$
shown in \textrm{Figure\;13}. 
However, it is open how many non-singular extensions a given map may admit.

\begin{figure}[t]
    \centering
    \includegraphics[width=70mm]{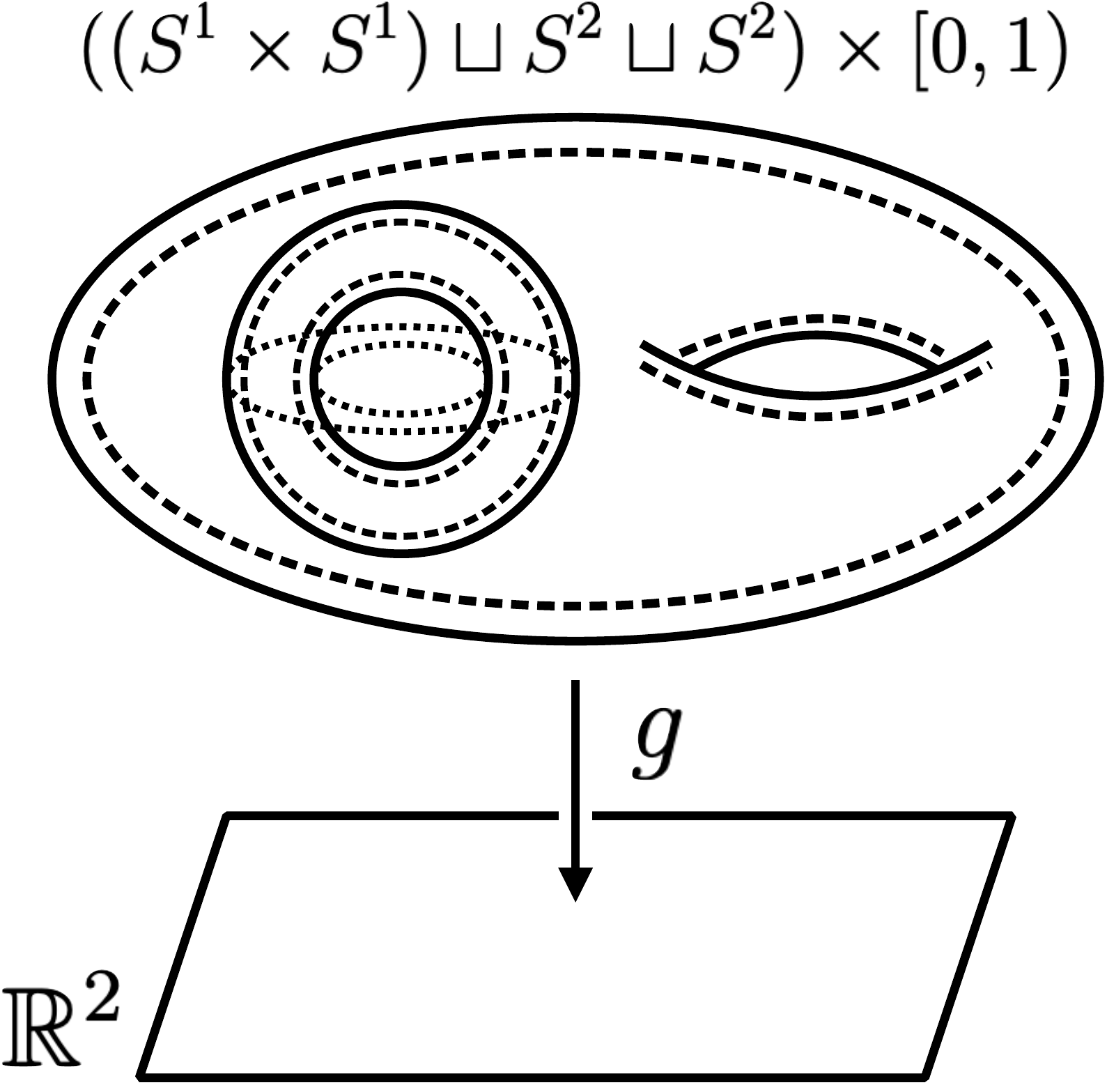}
    \caption{
    The submersion $g$. 
    }
\end{figure}

\begin{figure}[t]
    \centering
    \includegraphics[width=110mm]{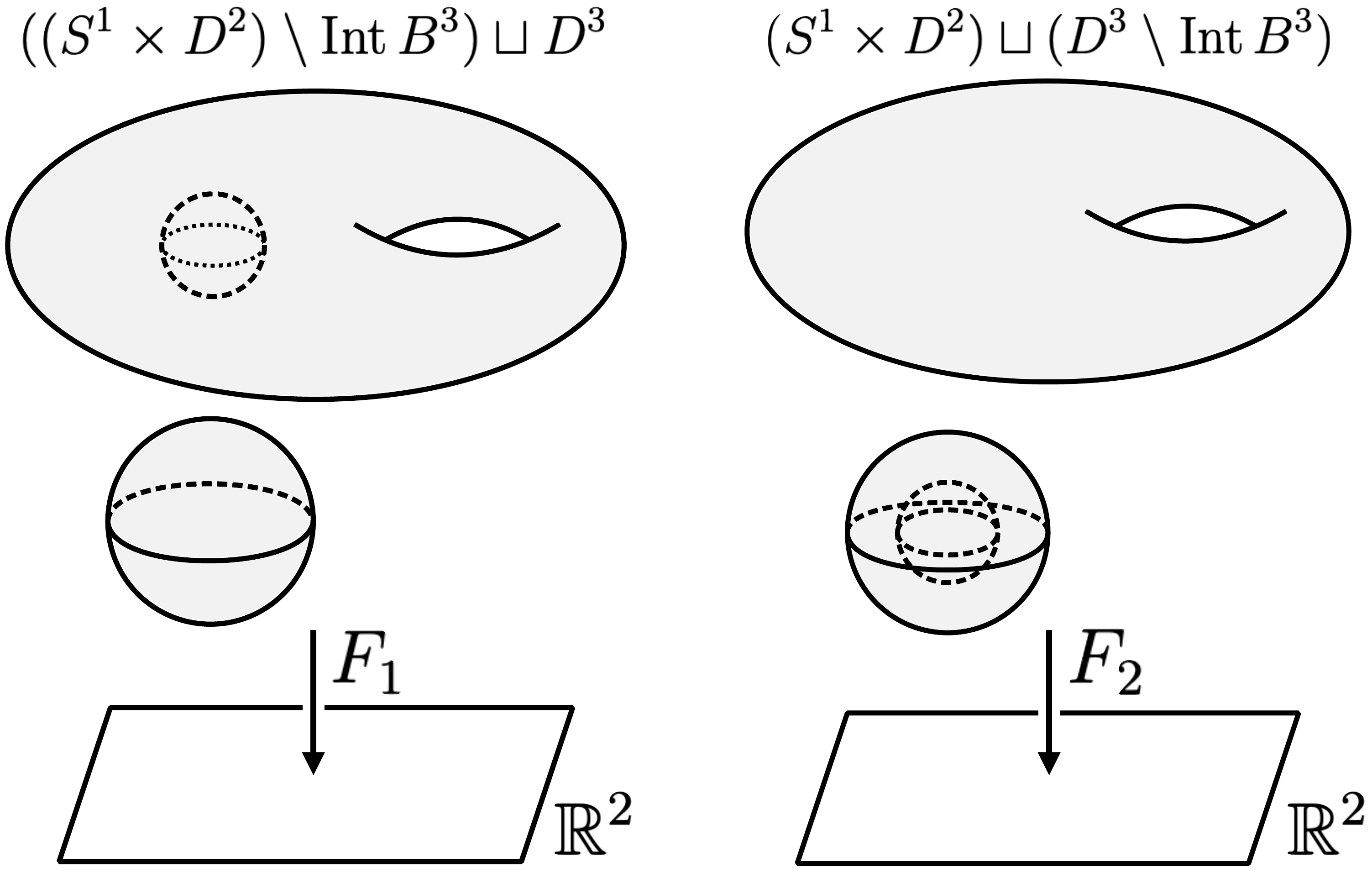}
    \caption{
    The non-singular extension $F_1$ and $F_2$ of $g$. 
    }
\end{figure}
\end{rem}

\section{Properties of the source manifolds of extensions}
In \textrm{Theorem\;4.1}, we obtained a necessary and sufficient condition for a given map to admit a non-singular extension. 
On the other hand, for applications, it is important to understand what properties such a non-singular extension possesses and what constraints its extension imposes on the original map. 
In this section, we present several results from these two perspectives. 
Recall that in this paper, a non-singular extension is a submersion that is a horizontal stable fold map on the boundary.
The results below concern such submersions.
We begin with a result on the Euler characteristic.

\begin{prop}
Let $N$ be a compact orientable $3$-dimensional manifold with non-empty boundary, and $F\colon N\to\mathbb{R}^2$ be a submersion that is a horizontal stable fold map on the boundary, with $F|_{S(F|_{\partial N})}$ having no double points. 
Then, we have
\begin{equation*}
\chi(N)=\sum_{r\in R\in\mathcal{R}_f}\chi(R)\chi(F^{-1}(r)).
\end{equation*}
\end{prop}

\proof
As stated in the proof of \textrm{Theorem\;4.1}, we can decompose $\mathbb{R}^2$ into the union of $K$ and $\bigsqcup_{R\in\mathcal{R}_f}L_{R}$.
By Ehresmann's fibration theorem\;\cite{E}, each map $F^{-1}(L_R)\to L_R$ is a fiber bundle for every $R\in\mathcal{R}_f$. 
By this decomposition, the manifold $N$ is decomposed as 
\begin{equation*}
N=F^{-1}(K)\cup\bigsqcup_{R\in\mathcal{R}_f}F^{-1}(R). 
\end{equation*}
By the local behavior of $F$ around each point of $S(F|_{\partial N})$ in \textrm{Proposition\;2.5}, we have $\chi(F^{-1}(K))=0$. 
Furthermore, since each $F^{-1}(R)$ is a fiber bundle, we obtain $\chi(F^{-1}(R))=\chi(R)\chi(F^{-1}(r))$ for any point $r\in R$. 
Moreover, each intersection $F^{-1}(K)\cap F^{-1}(R)$ is the disjoint union of fiber bundles over $S^1$ with the fiber $[-1,1]$ or $S^1$, both of which have $0$ as the Euler characteristic.
Therefore, applying the formula associated with the Euler characteristics of unions, we obtain the desired equation. 
\qed
\vspace{10pt}

We next present a result concerning the fundamental group. 
To enable its computation, we assume that $N$ is a compact connected $3$-dimensional manifold with boundary, that the submersion $F\colon N\to\mathbb{R}^2$ has no circle component in its fibers, and that the restriction $F|_{S(F|_{\partial N})}$ to the singular point set has no double points. 
\vspace{0pt}

\begin{defi}
We define the graph $H_F$ for $F$ as follows. 
We define its vertices by corresponding one vertex to each element of $\mathcal{R}_f$. 
We define its edge so that two vertices are joined when their corresponding elements in $\mathcal{R}_f$ are adjacent. 
Furthermore, each edge is labeled and directed as follows. 
We define each label with $\I$ or $\II$ according to whether the corresponding component of the singular point set of $F|_{\partial N}$ is type\;$\I$ or type\;$\II$. 
The direction is defined as the one in which the number of connected components of the fiber of $f$ increases. 
\end{defi}
We note that the graph $H_F$ is a tree, and recall that the vertices of degree one in trees are called \textit{leaves}.
We denote by $\mathcal{L}$ the set of leaves of $H_F$. 
\vspace{1pt}

Since $N$ is compact, there exists a unique non-bounded region $R_0\subset\mathbb{R}^2\setminus F(N)$, and we denote the vertex of $H_F$ corresponding to $R_0$ by $v_0$. 
We define the integers $e_1,e_2\geq 0$ as follows.   
From the vertex of $H_F$ corresponding to $R_0$, each leaf is reached via a simple path in $H_F$.
For each such path, we count the number of edges that are labeled with $\I$ and oriented in the same direction as the path.
We denote by $e_1$ the number of distinct edges that appear in the collection of all simple paths from the vertex corresponding to $R_0$ to each leaf, counting each edge only once, even if it appears in multiple paths.
Similarly, $e_2$ counts the distinct edges with $\II$, directed opposite to the path, that appear in all paths from the vertex corresponding to $R_0$ to the leaves.
We next state the result concerning the fundamental group.

\begin{prop}
Let $N$ be a compact connected orientable $3$-dimensional manifold with non-empty boundary. 
Let $F\colon N\to\mathbb{R}^2$ be a submersion without $S^1$ components in its fibers that is a horizontal stable fold map $f$ on the boundary without a double point of $f|_{S(f)}$. 
Then, we have the isomorphism
\begin{equation*}
\pi_1(N,x_0)\cong\mathbb{Z}^{r},
\end{equation*}
where 
\begin{equation*}
r=e_1-e_2-\sum_{R\in\mathcal{L}}\sharp F^{-1}(R).
\end{equation*}
\end{prop}

\proof
We decompose $N$ as in \textrm{Proposition\;5.1}.
By attaching $\bigsqcup_{R\in\mathcal{R}_f} F^{-1}(R)$ to $F^{-1}(K)$, we compute the fundamental group of $N$ using the Seifert--van Kampen theorem. 
Each attachment corresponds to an edge of the graph $H_F$. 
We divide the discussion into two cases, depending on whether the corresponding edge has a leaf, other than $v_0$, as an endpoint. 
First, we consider the case where the edge does not have a leaf, except for $v_0$, as an endpoint. 
In this case, by \textrm{Proposition\;2.5}, we know the behavior of the fundamental group according to the attachment explicitly: the rank changes by $e_1-e_2$. 
Next, we consider another case. 
Each element of $\mathcal{R}_f$ corresponding to such a leaf is homeomorphic to $D^2$, and its preimage under $F$ is homeomorphic to the disjoint union of $3$-disks. 
Thus, for the leaves, except for $v_0$, the change of the rank of the fundamental group is $\sum_{R\in\mathcal{L}}\sharp F^{-1}(R)$. 
From these discussions, the rank of the fundamental group of $N$ is given by 
$$r=e_1-e_2-\sum_{R\in\mathcal{L}}\sharp F^{-1}(R),$$ 
and hence $\pi_1(N,x_0)\cong\mathbb{Z}^r$, where $x_0\in N$ is a point. 
\qed
\vspace{10pt}

We next present a result from a different perspective.
It concerns the properties that a map may have if it admits a non-singular extension.
However, in the following corollary, we abuse the term ``non-singular extension" of $f$ to mean a submersion that bounds a map $f$.

\begin{cor}
Let $M$ be a closed orientable surface and $f\colon M\to\mathbb{R}^2$ be a horizontal stable fold map. 
Assume that $f$ admits a non-singular extension of a connected $3$-dimensional manifold whose fibers contain no $S^1$-components and for which $r=0$.
Then, $M$ is homeomorphic to the disjoint union of $2$-spheres.
\end{cor}

\proof
From \textrm{Proposition\;5.3} and the assumptions, the source manifold $N$ of the non-singular extension of $f$ is simply connected. 
Then, $N$ is homeomorphic to a closed 3-dimensional ball with finitely many open 3-dimensional balls removed.  
Therefore, the boundary $M=\partial N$ is homeomorphic to a disjoint union of finitely many $2$-spheres. 
\qed

\section*{Aknowledgement}
The author would like to thank Osamu Saeki for his useful discussions and comments. 
The author thanks Noriyuki Hamada and Naoki Kitazawa for their comments and support. 
This work has been partially supported by JSPS KAKENHI Grant Number JP23H05437 and by WISE program (MEXT) at Kyushu University.
Furthermore, we appreciate the referee's helpful comments that have improved this paper.

\end{document}